\documentclass[11pt]{amsart}

\usepackage{amscd,amsmath,amssymb,fancyhdr,color}
\usepackage[utf8]{inputenc}
\usepackage{amsfonts}
\usepackage{amsthm}
\usepackage{accents}
\usepackage{graphicx}
\usepackage{float}
\usepackage{subcaption}
\usepackage{verbatim}
\usepackage{indentfirst}
\usepackage{dsfont}
\usepackage{tikz}
\usepackage{tikz-cd}
\usetikzlibrary{matrix}
\usepackage[all]{xy}
\usepackage{enumerate}
\usepackage{extarrows}
\usepackage{csquotes}

\usepackage{scalerel,stackengine}
\stackMath
\newcommand\reallywidehat[1]{%
	\savestack{\tmpbox}{\stretchto{%
			\scaleto{%
				\scalerel*[\widthof{\ensuremath{#1}}]{\kern-.6pt\bigwedge\kern-.6pt}%
				{\rule[-\textheight/2]{1ex}{\textheight}}%WIDTH-LIMITED BIG WEDGE
			}{\textheight}% 
		}{0.5ex}}%
	\stackon[1pt]{#1}{\tmpbox}%
}

\usepackage{faktor}

\usepackage{BOONDOX-uprscr}

\usepackage[backref=page]{hyperref}
\renewcommand*{\backref}[1]{}
\renewcommand*{\backrefalt}[4]{%
	\ifcase #1 (Not cited.)%
	\or        (Cited on page~#2.)%
	\else      (Cited on pages~#2.)%
	\fi}

\hypersetup{
	colorlinks   = true,
	citecolor    = magenta
}

\newcommand{\K}{K\"ahler}

\DeclareMathOperator{\reg}{reg}
\DeclareMathOperator{\sing}{sing}
\DeclareMathOperator{\supp}{supp}

\DeclareMathOperator{\Ext}{Ext}

\numberwithin{equation}{section}

\def\eqref#1{(\ref{#1})}

\newcommand{\R}{{\mathbb R}}

\newcommand{\ei}{\textup{i}}

\newcommand{\del}{\partial}
\newcommand{\delb}{\overline{\partial}}

\def\1{\sqrt{-1}\:}

\newcommand{\cntrct}                % contraction with a vector field
{\hspace{2pt}\raisebox{1pt}{\text{$\lrcorner$}}\hspace{2pt}}

\newcommand{\Vol}{\operatorname{Vol}}
\newcommand{\Hom}{\operatorname{Hom}}

\newcommand{\codim}{\operatorname{codim}}

\newcommand{\Deck}{\operatorname{Deck}}

\newcommand{\ie}{{\em i.e. }}

\renewcommand{\to}{\longrightarrow}

\newcounter{Mycounter}[section]
\newcounter{lemma}[section]
\newcounter{claim}[section]
\newcounter{sublemma}[section]
\newcounter{corollary}[section]
\newcounter{theorem}[section]
\newcounter{conjecture}[section]
\newcounter{proposition}[section]
\newcounter{definition}[section]
\newcounter{example}[section]
\newcounter{remark}[section]
\newcounter{problem}[section]
\newcounter{question}[section]
\makeatletter

\@addtoreset{equation}{section}

\@addtoreset{footnote}{section}

\makeatother

\usetikzlibrary{arrows,chains,matrix,positioning,scopes}

\makeatletter
\tikzset{join/.code=\tikzset{after node path={%
			\ifx\tikzchainprevious\pgfutil@empty\else(\tikzchainprevious)%
			edge[every join]#1(\tikzchaincurrent)\fi}}}
\makeatother

\tikzset{>=stealth',every on chain/.append style={join},
	every join/.style={->}}

\makeatletter
\newtheorem*{rep@theorem}{\rep@title}
\newcommand{\newreptheorem}[2]{%
	\newenvironment{rep#1}[1]{%
		\def\rep@title{\ref{##1}}%
		\begin{rep@theorem}}%
		{\end{rep@theorem}}}
\makeatother

\newreptheorem{theorem}{Theorem}

\begin{document}
	
	\newpage
	
	\title[Blow-ups and modifications of lcK spaces]{Blow-ups and modifications of lcK spaces}
	
	\author{Ovidiu Preda}
	\address{Ovidiu Preda \newline
		\textsc{\indent University of Bucharest, Faculty of Mathematics and Computer Science\newline 
			\indent 14 Academiei Str., Bucharest, Romania\newline
			\indent and\newline
			\indent Institute of Mathematics ``Simion Stoilow'' of the Romanian Academy\newline 
			\indent 21 Calea Grivitei Street, 010702, Bucharest, Romania}}
	\email{ovidiu.preda@fmi.unibuc.ro; ovidiu.preda18@icloud.com}
	
	\author{Miron Stanciu}
	\address{Miron Stanciu \newline
		\textsc{\indent University of Bucharest, Faculty of Mathematics and Computer Science\newline 
			\indent 14 Academiei Str., Bucharest, Romania\newline
			\indent and\newline
			\indent Institute of Mathematics ``Simion Stoilow'' of the Romanian Academy\newline 
			\indent 21 Calea Grivitei Street, 010702, Bucharest, Romania}}
	\email{miron.stanciu@fmi.unibuc.ro; miron.stanciu@imar.ro}
	
	\thanks{Both authors were partially supported by a grant of Ministry of Research and Innovation, CNCS - UEFISCDI, project no.
		PN-III-P1-1.1-TE-2021-0228, within PNCDI III. \\\\[.2cm]
		{\bf Keywords:} Locally conformally \K, blow-up, proper modification.\\
		{\bf 2020 Mathematics Subject Classification:} 32S45; 53C55.
	}
	
	\date{\today}

	\begin{abstract}
		In this article, we prove that the blow-up of a locally irreducible lcK space $X$ along a subspace $Z$ which verifies certain conditions is lcK if and only if $X$ is induced gcK, generalizing a theorem of Ornea-Verbitsky-Vuletescu to singular locally irreducible spaces.  We also show that even if modifications of lcK spaces are not always of lcK type, they always admit quasi-lcK metrics.
	\end{abstract}
	
	\maketitle
	
	\hypersetup{linkcolor=blue}
	\tableofcontents

	\section{Introduction}
	
	In many problems arising in differential geometry, an essential step is choosing a good metric, in some sense determined by the particularities of the problem. Perhaps the best case scenario is that of K\" ahler metrics, which have been widely studied and have many useful properties. However, they are in a sense quite rare, as the existence of K\" ahler metrics on a given compact complex space has a number of topological obstructions. Therefore, in non-K\" ahler geometry, one seeks to replace them with a suitable class of Hermitian metrics. 
	
	In \cite{vaisman1976}, Vaisman introduced the notion of locally conformally K\" ahler (lcK) manifolds. They are complex manifolds endowed with a hermitian metric whose associated 2-form $\omega$ satisfies $d \omega=\theta\wedge \omega$ for a closed 1-form $\theta$, called the Lee form of $\omega$. This is equivalent to saying that locally, there exists a smooth function $f$ such that $e^{-f}\omega$ is K\" ahler, hence the name. The 1-form $\theta$ is exact if and only if the function $f$ can be defined globally, and in this case $\omega$ is called globally conformally K\" ahler (gcK). As an alternative definition, lcK manifolds can be described as quotients of K\" ahler manifolds by a discrete group of automorphisms which act as homotheties on the K\" ahler form. Some years later, Vaisman \cite{vaisman1980} proved that on a compact complex manifold, pure lcK (\ie non-gcK) and K\" ahler metrics with respect to the same complex structure cannot coexist. For a modern and comprehensive study of lcK geometry, see \cite{OV}.
	
	There is no straightforward definition of $(p, q)$-forms on complex spaces with singularities which preserves all the nice properties from the smooth case. Instead, Grauert \cite{grauert} generalized \K\ forms to complex spaces with possible singularities using families of locally defined strictly plurisubharmonic functions with some compatibility conditions. The same idea can be used to define lcK forms on complex spaces.
	
	\medskip
	
	In our first two papers \cite{PS1} and \cite{PS2}, we showed that both the characterization theorem of lcK manifolds \textit{via} the universal cover, and Vaisman's theorem on the \textit{pure lcK -- K\" ahler} dichotomy for compact complex manifolds remain true for complex spaces, the latter only for locally irreducible complex spaces (for the locally reducible case, we presented a counterexample). These two results are essential for the further study of lcK spaces. 
	
	In this paper, we study blow-ups, and, more generally, modifications of lcK spaces. Our first goal was to look at the main result of \cite{OVV}, which says that the blow-up of an lcK manifold $(M,\omega,\theta)$ along a compact complex submanifold $Z$ admits lcK metrics if and only if  $\omega_{\restriction Z}$ is gcK. We show that under reasonable conditions on $Z$, this result is still true for complex spaces. More exactly, we prove:
	
	\begin{reptheorem}{main_theorem_blowup}
		Let $(X,\omega,\theta)$ be an lcK space, and $Z\subset X$ a compact complex subspace, which is normal and is a locally complete intersection. 
		
		Then, the blow-up of $X$ along $Z$, denoted $\widehat{X}$, admits an lcK metric if and only if $\omega_{\restriction Z}$ is gcK.
	\end{reptheorem}
	
	In our proof for the direct implication, we make use of Varouchas' results \cite{Varouchas1989} which give sufficient conditions under which the image of a \K\ space under a holomorphic map is also of \K\ type. For this, we need $Z$ to be normal and the fibers of the canonical projection of the blow-up to be compact complex manifolds of equal dimension. The latter condition is satisfied if $Z$ is a locally complete intersection. We also use Vaisman's theorem for lcK spaces \cite{PS2}, for which we need local irreducibility of $Z$, guaranteed by the normality assumption. The reverse implication is true for any compact subspace $Z$, and this is done by refining \cite[Theorem 3.1]{PS2} to obtain the slightly more general \ref{blow-up-is-kahler-plus}.
	
	\medskip
	
	Since the blow-up of an lcK space is not necessarily of lcK type, we were then interested to find the closest class of metrics which are stable under blow-up and, more generally, under modifications. It turns out that this is achieved simply by working with a more general definition than that of lcK metrics and allowing the strictly plurisubharmonic functions which locally define the metric to take the value $-\infty$ on a controlled set of points. We call this type of metric a \textit{quasi-lcK metric}, inspired by the notion of \textit{quasi-K\" ahler metric} introduced by Col\c toiu \cite{Coltoiu1986}, and later used by Popa-Fischer \cite{PF} under the name of \textit{generalized K\" ahler metric}. The result we prove here is the following:
	
	\begin{reptheorem}{apf_with_lck}
		Let $p:X\to Y$ be a modification of the compact complex space $Y$. Suppose that $Y$ is quasi-lcK. Then, $X$ also admits a quasi-lcK metric. 
	\end{reptheorem}
	
	In particular, any modification of an lcK space is quasi-lcK. 
	
	The strategy for the proof is the following: we use \cite[Theorem 2.5]{PF} to construct a quasi-K\" ahler metric $\widetilde{\omega}$ on the universal cover of $\widetilde{X}$ of $X$. We show that if all the choices we make in that construction are well related, then $\Deck(\widetilde{X}/X)$ acts by homotheties on $\widetilde{\omega}$. Finally, \cite[Theorem 3.10]{PS1} can be easily adapted to quasi-lcK spaces, so it can be applied to conclude that $X$ admits a quasi-\K\ metric.
	
	\vspace{10pt}
	
	The paper is organized as follows: in Section \ref{sec:prelim} we give the definitions and the results we need about the blow-ups of complex spaces, modifications, K\" ahler, lcK, quasi-\K\ and quasi-lcK metrics, TC 1-forms, and also three important theorems for our goals: the first is a result about blow-ups of \K\ spaces, the second is Varouchas' theorem on the \K ianity of holomorphic images of \K \ spaces, and the third is Vaisman's theorem for lcK spaces; Section \ref{sec:blow-ups-of-lck-spaces} contains the proof of \ref{main_theorem_blowup}, with the main steps contained in two lemmas, for a better presentation; finally, Section \ref{sec:apf-cu-lck-sectiunea} contains the technical proof of \ref{apf_with_lck}.

	\section{Preliminaries}\label{sec:prelim}
	
	We begin this section by defining the blow-up of a complex space along a closed subspace.
	
	\begin{definition}\label{blow-up with graph}
		Let $X$ be a complex space and $Z=V(I)$ a closed subspace of $X$ defined by an ideal $I$ of $\mathcal{O}(X)$, generated by elements $g_0,g_1\ldots,g_k$. The morphism 
		\[
		\gamma:X\setminus Z \to \mathbb{P}^{k}, a\mapsto [g_0(a):g_1(a):\ldots:g_k(a)]
		\]
		is then well defined. The closure $\widehat{X}$ of the graph of $\gamma$ inside $X\times \mathbb{P}^{k}$ together with the restriction $\pi:\widehat{X}\to X$ of the projection $X\times\mathbb{P}^{k}\to X$ is the blow-up of $X$ along $Z$. 
		It does not depend, up to isomorphism over $X$, on the choice of the generators $g_i$ of $I$. Therefore, the blow-up along arbitrary closed subspaces can be constructed by gluing together local blow-ups. 
		$\pi^{-1}(Z)$ is a Cartier divisor, in particular a hypersurface, and is called the \textit{exceptional divisor} of the blow-up. $Z$ is called the \textit{center of the blow-up}.
	\end{definition}
	
	\begin{definition}\label{complete intersection}
		Let $X$ be a complex space and $Z\subset X$ a closed subspace. $Z$ is called \textit{complete intersection} in $X$ if there exists $g_1,\ldots,g_s\in\mathcal{O}(X)$, where $s=\codim_X Z$, such that $Z=\{x\in X \mid g_1(x)=\cdots=g_s(x)=0\}$. The closed subspace $Z$ is called \textit{locally complete intersection} if for any $x\in Z$, there exists an open subset $x\in U\subset X$, such that $Z\cap U$ is a complete intersection in $U$.
	\end{definition}
	
	\begin{remark}
		Let $X$ be a complex space and $Z\subset X$ a locally complete intersection, such that for an open $U\subset X$, we have $$Z\cap U=\{x\in U \mid g_0(x)=g_1(x)=\cdots=g_k(x)=0\},$$ with $g_0,g_1,\ldots g_k\in\mathcal{O}(U)$. Then, the blow-up $\widehat{X}$ of $X$ along $Z$ is defined above $U$ by
		\[
		\widehat{X}\cap \pi^{-1}(U)=\{(x,[z])\in U\times \mathbb{P}^k \mid g_i(x)z_j=z_i g_j(x), \text{ for any } 0\leq i,j\leq k\},
		\]
		where $\pi:\widehat{X}\to X$  is the canonical projection. Moreover, for any $x\in Z$, the fiber $\pi^{-1}(x)$ is isomorphic to $\mathbb{P}^k$. Also, restricting the projection $\pi$ to $$\pi_{\restriction \pi^{-1}(Z)}:\pi^{-1}(Z)\to Z,$$ we obtain a locally trivial holomorphic fibration, with $\mathbb{P}^k$ as fiber.
	\end{remark}
	
	\begin{definition}\label{proper_modification}
		A holomorphic map $p:X\to Y$ is called \textit{modification} if it is proper and there exists a rare analytic set $A\subset Y$ such that $p^{-1}(A)$ is rare in $X$ and such that $p_{\restriction {X\setminus p^{-1}(A)}}:X\setminus p^{-1}(A)\to Y\setminus A$ is a biholomorphism.
	\end{definition}
	
	\begin{remark}
		Blow-ups are particular cases of modifications.
	\end{remark}
	
	\bigskip
	
	The following are the main definitions we work with. It is customary in \K \ and non-\K \ smooth geometry to use the term ``metric" to refer also to the associated $2$-form, and we adopt the same convention:
	
	\begin{definition}\label{K+lcK}
		Let $X$ be a complex space. 
		\begin{enumerate}
			\item[\textbf{(K)}] A \textit{K\" ahler metric} on $X$ is the equivalence class $\reallywidehat{(U_i,\varphi_i)_{i\in I}}$ of a family such that $(U_i)_{i\in I}$ is an open cover of $X$, $\varphi_i:U_i\to \mathbb{R}$ is $\mathcal{C}^\infty$ and strictly psh, and $\ei\partial\overline{\partial}\varphi_i=\ei\partial\overline{\partial}\varphi_j$ on $U_i\cap U_j\cap X_{\reg}$, for every $i,j\in I$. 
			Two such families are equivalent if their union verifies the compatibility condition on the intersections, described above.
			
			\item[\textbf{(lcK)}] An \textit{lcK metric} on $X$ is the equivalence class $\reallywidehat{(U_i,\varphi_i,f_i)_{i\in I}}$ of a family such that $(U_i)_{i\in I}$ is an open cover of $X$, $\varphi_i:U_i\to \mathbb{R}$ is $\mathcal{C}^\infty$ and strictly psh, $f_i:U_i\to\mathbb{R}$ is smooth, and $\ei e^{f_i}\partial\overline{\partial}\varphi_i=\ei e^{f_j}\partial\overline{\partial}\varphi_j$ on $U_i\cap U_j\cap X_{\reg}$, for every $i,j\in I$. 
			Again, two such families are equivalent if their union verifies the compatibility condition mentioned above.
		\end{enumerate}
		
	\end{definition}
	
	\vspace{5pt}
	
	Next, we define quasi-K\" ahler and quasi-lcK metrics by allowing $-\infty$ as a value for the system of strictly plurisubharmonic functions. The definition for quasi-K\" ahler metrics was introduced in \cite{Coltoiu1986}, but we also require $\mathcal{C}^\infty$-regularity, as in \cite{PF}. 
	
	\begin{definition}\label{K_general+lcK_general}
		Let $X$ be a complex space. For $U\subset X$ and a psh function $\varphi:U\to\mathbb{R}$, we denote $\{\varphi=-\infty\}=A_{\varphi}$.
		\begin{enumerate}
			\item[\textbf{(q-K)}] A \textit{quasi-K\" ahler metric} on $X$ is the equivalence class $\reallywidehat{(U_i,\varphi_i)_{i\in I}}$ of a family such that 
			\begin{enumerate}
				\item[(a)] $(U_i)_{i\in I}$ is an open cover of $X$
				
				\item[(b)] $\varphi_i:U_i\to [-\infty,\infty)$ is strictly psh, $\varphi_i\not \equiv -\infty$ on any irreducible component of $U_i$, and $\varphi_i$ is of class $\mathcal{C}^\infty$ on $U_i\setminus A_{\varphi_i}$
				
				\item[(c)] $\ei\partial\overline{\partial}\varphi_i=\ei\partial\overline{\partial}\varphi_j$ on $(U_i\cap U_j)\setminus (X_{\sing}\cup A_{\varphi_i}\cup A_{\varphi_j})$, for any $i,j\in I$. 
				
				\item[(d)] $\varphi_i- \varphi_j$ restricted to $U_i\cap U_j \setminus (A_{\varphi_i}\cup A_{\varphi_j})$ is locally bounded around points of $A_{\varphi_i}\cup A_{\varphi_j}$,  for any $i,j\in I$. 
				
			\end{enumerate}
			
			\noindent Two such families are equivalent if their union still verifies the compatibility conditions (c) and (d).
			
			\vspace{5pt}
			
			\item[\textbf{(q-lcK)}] A \textit{quasi-lcK metric} on $X$ is the equivalence class $\reallywidehat{(U_i,\varphi_i,f_i)_{i\in I}}$ of a family $(U_i,\varphi_i,f_i)_{i\in I}$ such that $(U_i,\varphi_i)_{i\in I}$ verifies conditions (a) and (b) in \ref{K_general+lcK_general} -- (q-K), and moreover:
			\begin{enumerate}
				\item[(e)] $f_i:U_i\to \mathbb{R}$ is of class $\mathcal{C}^\infty$ for any $i\in I$
				
				\item[(f)] $\ei e^{f_i}\partial\overline{\partial}\varphi_i=\ei e^{f_j}\partial\overline{\partial}\varphi_j$ on $(U_i\cap U_j)\setminus (X_{\sing}\cup A_{\varphi_i}\cup A_{\varphi_j})$, for any $i,j\in I$. 
				
				\item[(g)] $(f_i-f_j)\varphi_i-\varphi_j$ restricted to $U_i\cap U_j \setminus (A_{\varphi_i}\cup A_{\varphi_j})$ is locally bounded around points of $A_{\varphi_i}\cup A_{\varphi_j}$,  for any $i,j\in I$. 
			\end{enumerate} 
			
			\noindent Two such families are equivalent if their union still verifies conditions (f) and (g).
		\end{enumerate}
		
	\end{definition}
	
	\begin{remark}
		The definition given in \cite{PF} for quasi-\K\ metrics is stronger than the conditions required by our \ref{K_general+lcK_general} -- (q-K), and they coincide for normal spaces. 
	\end{remark}
	
	\vspace{10pt}
	
	The following definition allows us to work with a substitute of the Lee form of a quasi-lcK metric, even if we are not in the smooth context: 
	
	\begin{definition}\label{TC-1-form-definition}
		\begin{itemize}
			\item Let $X$ be a topological space and consider $(U_i,f_i)_{i\in I}$, consisting of an open cover $(U_i)_{i\in I}$ of $X$ and a family of continuous functions $f_i:U_i\to\mathbb{R}$ such that $f_i-f_j$ is locally constant on $U_i\cap U_j$, for all $i,j\in I$. The class  
			\[
			\theta =\reallywidehat{(U_i,f_i)_{i\in I}}\in  \check{\mathrm{H}}^0\left(X,\faktor{\mathscr{C}}{\underline{\R}}\right)
			\]
			is called a \textit{topologically closed $1$-form} (\textit{TC $1$-form}). 
			\item We say that a TC 1-form $\theta$ is \textit{exact} if $\theta = \widehat{(X, f)}$ for a continuous function $f:X \to \mathbb{R}$. In this case, we make the notation $\theta = df$.
			\item 	Let $\omega=\reallywidehat{(U_i,\varphi_i,f_i)_{i\in I}}$ be an lcK metric on a complex space $X$. Then, the TC 1-form $\theta=\reallywidehat{(U_i,f_i)_{i\in I}}$ is called the \textit{Lee form} of $\omega$. If $\theta$ is exact, then $\omega$ is called \textit{globally conformally K\" ahler (gcK)}. We have the obvious analogous definition for \textit{q-gcK}.
		\end{itemize}
	\end{definition}
	
	\vspace{10pt}
	
	An important result in the geometry of \K \ spaces by Fujiki (see \cite[3.1]{Varouchas1986}, \cite[Lemma 2]{fuji75}) states that the blow-up of a compact \K \ space along a complex subspace is also \K. A rewritten proof can be found in our previous paper \cite[Theorem 3.1]{PS2}. Moreover, a careful examination of that proof shows that we can obtain a little more by observing that the line bundle $\mathcal{O}(1)$ is trivial outside any neighborhood of the exceptional divisor and choosing conveniently the sections involved in the construction of the new metric, such that this new metric coincides with the pull-back of the old one away from the exceptional divisor. This also allows us to drop the compactness requirement on $X$. Hence, we obtain the following slightly improved version:
	
	\begin{theorem}\label{blow-up-is-kahler-plus}
		Let $(X,\omega)$ be a K\"aher space and $Z\subset X$ a compact complex subspace of positive codimension. Let $V\supset Z$ be an open neighborhood. Then, the blow-up $\pi:\widehat{X}\to X$ of $X$ along $Z$ admits a K\"ahler metric $\widehat{\omega}$ such that $\widehat{\omega}=\pi^*\omega$ on $\widehat{X}\setminus \widehat{V}$, where $V=\pi^{-1}(V)$.
	\end{theorem}
	
	\vspace{10pt}
	
	By combining \cite[Proposition 3.3.1]{Varouchas1989} and \cite[Theorem 3$^\prime$]{Varouchas1989}, we get the following theorem, which gives sufficient conditions under which the image of a \K\ space under a holomorphic map is also of \K\ type.
	
	\begin{theorem}\label{th_varouchas}
		Let $(X,\omega)$ be a \K\ space, $X^\prime$ a normal space, and $\pi:(X,\omega)\to X^\prime$ a holomorphic function with the following properties:
		\begin{enumerate}
			\item[(i)] $\pi$ is proper, open, and surjective.
			\item[(ii)] all fibers of $\pi$ are of pure dimension $m$.
		\end{enumerate}
		Then, $X^\prime$ also admits a \K\ metric.
	\end{theorem}
	
	\vspace{10pt}
	
	Vaisman's theorem \cite{vaisman1980}, a fundamental result of lcK geometry, states that on a compact complex manifold, \textit{pure lcK} and \textit{K\" ahler} metrics (with respect to the same complex structure) cannot coexist. A generalization of Vaisman's theorem to locally irreducible complex spaces is \cite[Theorem 4.4]{PS2}, enounced below.
	
	\begin{theorem}\label{th_vaisman_singular}
		Let $(X,\omega,\theta)$ be a compact, locally irreducible, lcK space. If $X$ admits a K\"ahler metric, then $(X,\omega,\theta)$ is gcK. 
	\end{theorem}

	\section{Blow-ups of lcK spaces}\label{sec:blow-ups-of-lck-spaces}
	
	In this section we prove one of our main results, giving necessary and sufficient conditions for a blow-up of an lcK space to inherit an lcK structure. The essential ingredient in the proof is the following adaptation to our context of \cite[Lemma 3.1]{OVV} on fibrations.
	
	\begin{lemma}\label{lemma_on_fibrations}
		Let $\pi:X\to Z$ be a locally trivial fibration with fiber $F$, where $(X,\omega,\theta)$ is an lcK space, $Z$ is a locally irreducible complex space, and the fiber $F$ is a connected compact complex manifold of positive dimension. Assume also that the map $\pi^*:H^1(Z)\to H^1(X)$ is an isomorphism. Then, $\omega$ is gcK.
	\end{lemma}
	\begin{proof}
		Let $\widetilde{X}$ be the minimal covering $\rho_X:\widetilde{X}\to X$ such that $\rho_X^*(\theta)$ is exact. Since $\pi^*:H^1(Z)\to H^1(X)$ is an isomorphism, there exists a covering $\rho_Z:\widetilde{Z}\to Z$ such that the diagram 
		\[
		\xymatrix{
			\widetilde{X} \ar[d]_-{\widetilde{\pi}} \ar[r]^-{\rho_X} & X \ar[d]^-{\pi}\\
			\widetilde{Z} \ar[r]_-{\rho_Z} & Z 
		}
		\]	
		\noindent is commutative, where $\widetilde{\pi}:\widetilde{X}\to\widetilde{Z}$ is also a locally trivial fibration with fiber $F$. Denote by $\widetilde{\omega}$ the K\" ahler form on $\widetilde{X}$ and $F_{\tilde{b}}=\widetilde{\pi}^{-1}(\tilde{b})\simeq F$. Also, denote $\widetilde{X}_{\text{r}}=\widetilde{\pi}^{-1}(\widetilde{Z}_{\reg})$, which is a complex manifold. $Z$ is locally irreducible, hence so is $\widetilde{Z}$, which implies that $\widetilde{Z}_{\reg}$ is connected, therefore $\widetilde{X}_{\text{r}}$ is connected too.
		By the Universal Coefficient Theorem for cohomology, we have 
		\[
		0 \rightarrow \Ext_{\mathbb{R}}(H_{2k-1}(\widetilde{X}_{\text{r}},\mathbb{R}),\mathbb{R}) \rightarrow  H^{2k}(\widetilde{X}_{\text{r}},\mathbb{R}) \rightarrow  \Hom_{\mathbb{R}}(H_{2k}(\widetilde{X}_{\text{r}},\mathbb{R}),\mathbb{R}) \rightarrow  0
		\]
		However, since $\Ext_{\mathbb{R}}(H_{2k-1}(\widetilde{X}_{\text{r}},\mathbb{R}),\mathbb{R})=0$ and $\widetilde{\omega}^k$ determines a class in $H^{2k}(\widetilde{X}_{\text{r}},\mathbb{R})$, we obtain that 
		$$F_{\tilde{b}}\mapsto \int_{F_{\tilde{b}}}\widetilde{\omega}^k=\Vol_{\widetilde{\omega}}(F_{\tilde{b}})$$
		depends only on the homology class $[F_{\tilde{b}}]\in H_{2k}(\widetilde{X}_{\text{r}},\mathbb{R})$.
		Also, since $\widetilde{Z}_{\reg}$ is connected, for any two base points $\tilde{b}_1, \tilde{b}_2\in \widetilde{Z}_{\reg}$, we have $[F_{\tilde{b}_1}]=[F_{\tilde{b}_2}]$, thus $\Vol_{\widetilde{\omega}}(F_{\tilde{b}_1})=\Vol_{\widetilde{\omega}}(F_{\tilde{b}_2})$. Consequently,
		\begin{equation*}
			\begin{split}
				\Vol_{\widetilde{\omega}}(F_{\tilde{b}})&=\int_{F_{\tilde{b}}}\widetilde{\omega}^k=\int_{F_{\gamma^{-1}(\tilde{b})}}(\gamma^*\widetilde{\omega})^k=\int_{F_{\gamma^{-1}(\tilde{b})}}(c_\gamma\widetilde{\omega})^k =c_\gamma^k\Vol_{\widetilde{\omega}}(F_{\gamma^{-1}(\tilde{b})})\\
				&=c_\gamma^k\Vol_{\widetilde{\omega}}(F_{\tilde{b}}),
			\end{split}
		\end{equation*}
		hence $c_\gamma=1$ for all $\gamma\in \Gamma=\Deck(\widetilde{X}/X)$. This means that $\widetilde{\omega}$ is $\Gamma$-invariant, and $\omega$ is gcK.
	\end{proof}
	
	\begin{lemma}\label{lemma_blow-up_IGCK}
		Let $(X,\omega,\theta)$ be an lcK space and $Z\subset X$ a compact complex subspace, such that $\omega_{\restriction Z}$ is gcK. Then, the blow-up $\pi:\widehat{X}\to X$ of $X$ along $Z$, admits an lcK metric.
	\end{lemma}
	\begin{proof}
		Since $Z$ admits triangulations (see \cite{KB} and \cite{LW}), it is a CW-complex and so, by \cite{cauty}, an absolute neighborhood retract. Since $\omega$ is gcK on $Z$, there then exists a neighborhood $U \supset Z$ such that $\theta_{\restriction U}$ is exact. Possibly by restricting $U$, we may assume that there exists a globally defined $f:X \to \R$ such that $\theta_{\restriction U} = df_{\restriction U}$ or, equivalently, $e^{-f} \omega$ is a \K \ form on $U$. Next, choose an open neighborhood $V$ such that $Z\subset V\Subset U$ and denote $\widehat{U}=\pi^{-1}(U)$ and $\widehat{V}=\pi^{-1}(V)$. By \ref{blow-up-is-kahler-plus}, there exists a \K \ metric $\omega_{\widehat{U}}$ on $\widehat{U}$ such that $\omega_{\widehat{U}}=\pi^*(e^{-f}\omega)$ on $\widehat{U}\setminus \widehat{V}$. Finally, $\omega_{\widehat{U}}$ and $\pi^*(e^{-f}\omega)_{\restriction \widehat{X}\setminus \widehat{V}}$ can be glued together to get an lcK metric on $\widehat{X}$.
	\end{proof}
	
	\medskip
	
	We can now prove
	
	\begin{theorem}\label{main_theorem_blowup}
		Let $(X,\omega,\theta)$ be an lcK space, and $Z\subset X$ a compact complex subspace, which is normal and is a locally complete intersection. 
		
		Then, the blow-up of $X$ along $Z$, denoted $\widehat{X}$, admits an lcK metric if and only if $\omega_{\restriction Z}$ is gcK.
	\end{theorem}
	\begin{proof}
		The ``if" part follows directly from \ref{lemma_blow-up_IGCK}, hence it is true even without the additional assumptions on $Z$.
		
		For the ``only if" part, we consider the restriction $\pi_{\restriction \pi^{-1}(Z)}:\pi^{-1}(Z)\to Z$, which is a locally trivial fibration, and verifies all the conditions in \ref{lemma_on_fibrations}. Therefore, $\pi^{-1}(Z)$ is gcK. Then, by \ref{th_varouchas}, $Z$ admits a \K\ metric. $Z$ is normal, in particular locally irreducible, so Vaisman's theorem for lcK spaces (\ref{th_vaisman_singular}) ensures that the lcK metric $\omega_{\restriction Z}$ is gcK.
	\end{proof}
	
	\section{Modifications of lcK spaces}\label{sec:apf-cu-lck-sectiunea}
	
	We now turn to the second goal of our paper, which is to find a type of non-\K \ structure which is stable under modifications. As proved by \cite[Theorem 2.9]{OVV} and \ref{main_theorem_blowup}, lcK structures do not have this property even in the smooth setting. This is instead achieved by looking at quasi-lcK metrics. The result below generalises \cite{PF}:
	
	\begin{theorem}\label{apf_with_lck}
		Let $p:X\to Y$ be a modification of the compact complex space $Y$. Suppose that $Y$ is quasi-lcK. Then, $X$ also admits a quasi-lcK metric. 
	\end{theorem}
	\begin{proof}
		We assume for this proof that $Y$ is not only quasi-lcK, but lcK. The general case requires only minor changes when writing the compatibility conditions and is left to the reader.
		
		Denote by $\pi_Y : Y_0 \to Y$ the universal covering of $Y$, let $\Gamma = \Deck(Y_0/Y)$, and consider $X_0=X\times_Y Y_0$  to be the pull-back of the universal cover $\pi_Y:Y_0\to Y$ along $p:X\to Y$. Then, we have the following commutative diagram:
		\[
		\xymatrix{
			X_0 \ar[d]_-{\pi_X} \ar[r]^-{p_0} & Y_0 \ar[d]^-{\pi_Y}\\
			X \ar[r]_-{p} & Y
		}
		\]		
		\noindent where $p_0:X_0\to Y_0$ is a modification, and $\pi_X:X_0\to X$ is a covering of $X$.
		Consider $\omega=\reallywidehat{(V_j,\lambda_j,f_j)_{j\in I}}$ an lcK metric on $Y$ such that $I$ is finite, the open sets $V_j$ are all connected and Stein, and for every $j\in I$, 
		\[
		\pi_Y^{-1}(V_j)=\bigcup_{\gamma\in\Gamma} V_j^\gamma
		\]
		is a disjoint union of copies of $V_j$, such that for any $\eta\in\Gamma$, $\eta(V_j^\gamma)=V_j^{\eta\gamma}$.
		The proof of \cite[Theorem 3.10]{PS1} shows that there exists a smooth function $g:Y_0\to\mathbb{R}$ such that $\omega_0:=e^{-g}\pi_Y^*\omega$ is a K\" ahler metric on $Y_0$, and such that $\Gamma$ acts on $\omega_0$ by positive homotheties, \textit{i.e.} for every $\gamma\in\Gamma$, $\gamma^*\omega_0=c_\gamma \omega_0$, where $c_\gamma>0$.
		
		Note that the above imply that $\pi_Y^* f_j - g$ is locally constant on $\pi_Y^{-1}(V_j)$, so
		\begin{equation}
			\label{eq:apflck1}
			\pi_Y^* f_j - g = d_j^\gamma \in \R \text{ on every } V_j^\gamma,
		\end{equation}
		and, furthermore, that for $\eta \in \Gamma$, $\eta^* g = g - \ln c_\gamma$. Applying $\eta^*$ to \eqref{eq:apflck1}, we obtain
		\begin{equation}
			\label{eq:apflck2}
			d_j^{\eta \gamma} = d_j^\gamma + \ln c_\eta, \ \forall \eta \in \Gamma.
		\end{equation}
		
		It will be useful later to use the compatibility property for the local potentials of the \K \ form  $\omega_0$ on $Y_0$. On $V_j^\gamma\setminus (Y_0)_{\sing}$, we have
		\[
		\omega_0 = e^{-g} \pi^* ( e^{f_j} \ei \del \delb \lambda_j ) = \ei \del \delb ( e^{d_j^\gamma} \lambda_j ),
		\]
		so 
		\begin{equation}
			\label{eq:apflck3}
			\ei \del \delb ( e^{d_j^\gamma} \lambda_j ) =\ei \del \delb ( e^{d_i^\eta} \lambda_i ) \text{ on } (V_j^\gamma \cap V_i^\eta) \setminus (Y_0)_{\sing}.
		\end{equation}
		
		We have the open covering of $X_0$ given by  $U_j^\gamma=p_0^{-1}(V_j^\gamma)$ with $j\in I, \gamma\in \Gamma$. For a fixed $j\in I$, the open sets $(U_j^\gamma)_\gamma$ are mutually disjoint. Also, on each $U_j^\gamma$ we have the function $\lambda_j\circ p_0$, which is psh on $U_j^\gamma$, but not necessarily strictly psh. 
		
		Since $V_j$ are Stein, we may assume that we have a covering $(V'_j)_{j \in I}$ with $V'_j \Subset V_j$ embedded as analytic sets in open balls $\Phi_j: V'_j \to B(0, r_j)$ such that $0 \in \Phi_j(V'_j)$ and $\Phi_j(z) \to \partial B(0, r)$ if $x \to \partial V'_j$. As before, for every $j\in I$, we have
		$
		\pi_Y^{-1}(V^\prime_j)=\bigcup_{\gamma\in\Gamma} V_j^{\prime\gamma}
		$
		a disjoint union of copies of $V^\prime_j$, such that for any $\eta\in\Gamma$, $\eta(V_j^{\prime\gamma})=V_j^{\prime\eta\gamma}$.
		
		As $p:X\to Y$ is a modification, by definition there exists a rare analytic set $A\subset Y$ such that $p^{-1}(A)$ is rare in $X$ and $p_{\restriction{X\setminus p^{-1}(A)}}:X\setminus p^{-1}(A)\to Y\setminus A$ is a biholomorphism. 
		
		Next, using the same arguments as in \cite[First step of the proof, pp.841-844]{PF}, there exists a coherent sheaf of ideals $\mathcal{I}$ on $Y$ of holomorphic functions which vanish on $A$, and for every $j\in I$, there exist sections $f_{j,k}\in\mathcal{I}(V_j), k=1,\ldots,s_j$, generating each fiber of $\mathcal{I}$ over $V_j$, such that 
		\[
		A\cap V_j=\{ x\in V_j | f_{1,j}(x)=\cdots=f_{j,s_j}(x)=0 \}, 
		\]
		and moreover, such that if we consider
		\begin{equation*}
			\label{eq:psi}
			\psi_j=\lambda_j + \log\left(\sum_{l=1}^{s_j} |f_{j,l}|^2\right),
		\end{equation*}
		then $\psi_j\circ p$ is strictly psh on $U_j=p^{-1}(V_j)$. 
		
		Denoting $\psi_j^\gamma=e^{d_j^\gamma}(\psi_j\circ \pi_Y)_{|V_j^\gamma}$, 
		we have $\psi_j^\gamma\circ p_0$ strictly psh on $U_j^\gamma$. However, we cannot use this family of functions right away to construct a \K\ metric on $X_0$, because the \enquote{perturbation term} $\log\left(\sum_{l=1}^{s_j} |f_{j,l}|^2\right)$ in the definition of $\psi_j$ has ruined the compatibility condition. 
		For this reason, we turn to ideas in \cite[Second step of the proof, pp.844-846]{PF}, to construct a globally defined function $\mathfrak{v}$ on $X_0$, which we use to restore the compatibility condition required for a \K\ metric. 
		
		For any $j\in I$, consider  
		$$a_j:=|f_{j,1}|^2+\cdots+|f_{j,s_j}|^2 \text{ on } V_j.$$
		
		On $V_j\cap V_k$, the sections in $\mathcal{I}(V_j\cap V_k)$ are generated by $(f_{j,1},\ldots,f_{j,s_j})_{\restriction V_j\cap V_k}$, and also by $(f_{k,1},\ldots,f_{k,s_k})_{\restriction V_j\cap V_k}$, hence
		the quotient $$\frac{a_j}{a_k}=\frac{|f_{j,1}|^2+\cdots+|f_{j,s_j}|^2}{|f_{k,1}|^2+\cdots+|f_{k,s_k}|^2}$$ is lower and upper bounded on $(V_j^\prime\cap V_k^\prime)\setminus A$, hence $\log a_j - \log a_k$ is bounded on $(V_j^\prime\cap V_k^\prime)\setminus A$. 
		
		Now consider $\mathfrak{v}_j:V_j^\prime\to[-\infty,\infty)$ defined by
		\[
		\mathfrak{v}_j(z)=\log a_j(z)-\frac{1}{r_j^2-|\Phi_j(z)|^2}=:\log a_j(z)-\sigma_j(z).
		\]
		By construction, $$\mathfrak{v}_j(z)\to -\infty  \text{ for } z\to \partial V_j^\prime, \text{ and } \mathfrak{v}_j(z)= -\infty \text{ for } z\in A\cap V_j^\prime.$$
		For each $\gamma\in\Gamma$, denote $\mathfrak{v}_j^\gamma=e^{d_j^\gamma}(\mathfrak{v}_j\circ\pi_Y)_{\restriction V_j^{\prime\gamma}}$.
		We now make use of Demailly's technique of regularized maximum \cite[Lemma 5.18, p.43]{Demailly_book} to glue the functions $\mathfrak{v}_j^\gamma$ to a function $\mathfrak{v}$ on $Y_0$. Let $\rho:\mathbb{R}\to\mathbb{R}$ be a smooth function with $\rho\geq 0$, $\supp\rho\subset \left[-\frac{1}{2},\frac{1}{2}\right]$, such that $\int_\mathbb{R} \rho(u) du=1$ and consider the function $m:\mathbb{R}^q\to\mathbb{R}$ given by
		\[
		m(t_1,\ldots,t_q)=\int_{\mathbb{R}^q}\max\{t_1+u_1,\ldots,t_q+u_q\}\prod_{1\leq n\leq q}\rho(u_n)du_n.
		\]
		It is easy to verify that whenever $$t_j<\max\{t_1,\ldots,t_{j-1},t_{j+1},\ldots,t_q\}-1,$$ we have $$m(t_1,\ldots,\widehat{t_j},\ldots,t_q)=m(t_1,\ldots,t_j,\ldots,t_q),$$ where $\widehat{\cdot}$ denotes, as usual, that the respective variable is missing.		
		
		Note that by construction, our covering $(V_j^{\prime\gamma})_{j\in I,\gamma\in \Gamma}$ of $Y_0$ has the property that any point $z\in Y_0$ belongs to finitely many sets in the covering, more precisely, at most $q=|I|$ sets. 
		Thus, it makes sense to consider $\mathfrak{v}(z)=m((\mathfrak{v}_j^\gamma(z))_{j,\gamma})$, since by ignoring the $\mathfrak{v}_j^\gamma$'s for which $z\not \in V_j^{\prime\gamma}$, we are left with a finite family of functions as arguments for $m$. 
		
		Moreover, since $\mathfrak{v}_j^\gamma(z)\to -\infty$ as $z\to\partial V_j^{\prime\gamma}$, by our previous remark, the values of $\mathfrak{v}_j^\gamma(z)$ near the boundary of $V_j^{\prime\gamma}$ do not play an effective role in the maximum. 
		We can thus consider a covering $(W_j)_j$ of $Y$ such that $W_j \Subset V_j^\prime$ and $\mathfrak{v}_j$ for $z \in V^{\prime}_j \setminus W_j$ does not have an effective role in $m$, \ie $m(({\mathfrak{v}_i}_{\restriction W_j})_i) = m((\mathfrak{v}_i)_i).$
		For every $j\in I$, we define as before
		$
		\pi_Y^{-1}(W_j)=\bigcup_{\gamma\in\Gamma} W_j^\gamma
		$
		such that $\eta(W_j^\gamma)=W_j^{\eta\gamma}$ for every $\eta\in\Gamma$. 
		Note that our assumption for $W_j$ implies that, for every $\gamma \in \Gamma$, $\mathfrak{v}_j^\gamma$ for $z\in V_j^{\prime\gamma}\setminus W_j^\gamma$ does not have an effective role in $m$, that is  $m(({\mathfrak{v}_j^\gamma}_{\restriction W_j^\gamma})_{j,\gamma})=m((\mathfrak{v}_j^\gamma)_{j,\gamma})$.
		
		As $\lambda_j$ is strictly psh on $V_j$ and $\sigma_i$ and all its derivatives are bounded on $W_i$, we can find a big enough $M_{i, j}$ such that $M_{i, j} \lambda_j - \theta_i$ is strictly psh on $V_j \cap W_i$. Take $M = \max \{ M_{i, j} \ | \ i, j \in I \}$. 
		
		On $p_0^{-1} ( V_j^\gamma \cap W_i^\gamma ) = U_j^\gamma \cap p_0^{-1} (W_i^\gamma)$, we can define for each $\gamma \in \Gamma$ the function 
		\[
		\varphi_j^\gamma := \left( M e^{d_j^\gamma} \lambda_j \circ \pi_Y + \mathfrak{v} \right) \circ p_0 = m \left( \left( M e^{d_j^\gamma} \lambda_j \circ \pi_Y \circ p_0 + \mathfrak{v}_i^\eta \circ p_0 \right)_{i \in I, \eta \in \Gamma} \right).
		\]
		Note that $$M e^{d_j^\gamma} \lambda_j \circ \pi_Y \circ p_0 + \mathfrak{v}_i^\eta \circ p_0 = e^{d_j^\gamma} \left( M\lambda_j \circ \pi_Y \circ p_0 + \log a_i \circ p_0 - \theta_i \circ p_0 \right). $$
		
		Recall that on $U_i^\gamma$ we have defined the function 
		\[
		\psi_i^\gamma \circ p = e^{d_i^\gamma} \psi_i \circ \pi_Y \circ p_0 = e^{d_i^\gamma} \left(\lambda_i + \log a_i \right) \circ \pi_Y \circ p_0
		\]
		which was strictly psh. However, the only requirement for this was that $\lambda_i$ is strictly psh on $V_i$, so we may replace it by $M\lambda_j -\theta_i$ and obtain that 
		\[
		e^{d_j^\gamma} \left( M\lambda_j - \theta_i + \log a_i \right) \circ \pi_Y \circ p_0 = e^{d_j^\gamma} \left( M\lambda_j \circ \pi_Y \circ p_0 + \log a_i \circ p_0 - \theta_i \circ p_0 \right)
		\]
		is strictly psh on $U_j^\gamma \cap p_0^{-1} (W_i^\gamma)$, for all $i, j \in I$ and all $\gamma \in \Gamma$. 
		
		Now, the properties of the regularized maximum function $m$, ensure that $m(\tau_1, \tau_2, ..., \tau_k)$ is strictly psh for any strictly psh functions $\tau_1, ..., \tau_k$ (see \cite[Lemma 5.18]{Demailly_book} and \cite[p.846]{PF}). Since being strictly psh is a local property, this remains true for a family of strictly psh functions with locally finite domains, such that each of them tends to $-\infty$ at the boundary, as is the case for the definition of $\varphi_j^\gamma$. By all the arguments above we get that 
		\[
		\varphi_j^\gamma: U_j^\gamma \to [-\infty, \infty), \ \varphi_j^\gamma = \left( Me^{d_j^\gamma} \lambda_j \circ \pi_Y + \mathfrak{v} \right) \circ p_0 
		\]
		is strictly psh on $U_j^\gamma$ and is regular outside of $$(\varphi_j^\gamma)^{-1}(\{-\infty \}) = U_j^\gamma \cap (\pi_Y \circ p_0)^{-1}(A),$$ for all $j \in I$, $\gamma \in \Gamma$.
		
		We claim that this family of strictly psh functions $(\varphi_j^\gamma)_{j \in I, \gamma \in \Gamma}$ on the covering $(U_j^\gamma)_{j \in I, \gamma \in \Gamma}$ define a quasi-\K \ metric on $X_0$ \ie that they also satisfy conditions $(c)$ and $(d)$ from \ref{K_general+lcK_general}.

		Indeed, using \eqref{eq:apflck3}, on 
		\[ 
		(U_j^\gamma \cap U_i^\eta) \setminus ((X_0)_{\textrm{sing}} \cup A_{\varphi_j^\gamma}  \cup A_{\varphi_i^\eta}) = (U_j^\gamma \cap U_i^\eta) \setminus ((X_0)_{\textrm{sing}} \cup (\pi_Y \circ p_0)^{-1}(A) ),
		\] 
		we get
		\begin{align*}
			\ei \del \delb (\varphi_j^\gamma - \varphi_i^\eta) &= \ei \del \delb \left( M e^{d_j^\gamma} \lambda_j \circ \pi_Y \circ p_0 - M e^{d_i^\eta} \lambda_i \circ \pi_Y \circ p_0 \right) \\
			&= \ei M \left( \del \delb (e^{d_j^\gamma} \lambda_j) - \del \delb (e^{d_i^\eta} \lambda_i ) \right) \circ \pi_Y \circ p_0 \\
			&= 0.
		\end{align*}
		
		Furthermore, $\varphi_j^\gamma - \varphi_i^\eta$ on $(U_j^\gamma \cap U_i^\eta) \setminus (A_{\varphi_j^\gamma}  \cup A_{\varphi_i^\eta})$ is locally bounded around points of $A_{\varphi_j^\gamma}  \cup  A_{\varphi_i^\eta}$ because $\lambda_i - \lambda_j$ are locally bounded.
		
		Having proven that $\omega_{1} := \reallywidehat{ (U_j^\gamma, \varphi_j^\gamma)_{j \in I, \gamma \in \Gamma} }$ is a quasi-\K \ metric on $X_0$, \cite[Theorem 3.10]{PS1} can be easily adapted to show that this implies $X$ has a quasi-lcK structure if $\Deck(X_0/X) \simeq \Gamma$ acts by positive homotheties on $\omega_{1}$. Indeed, for any $\eta \in \Gamma$, on $U_j^\gamma \setminus ((X_0)_{\textrm{sing}} \cup (\pi_Y \circ p_0)^{-1}(A))$, we have
		\begin{align*}
			\begin{split}
				\eta^* \omega_1 &= \ei \del \delb ( \varphi_j^{\eta \gamma} \circ \eta ) = \ei \del \delb ( (M e^{d_j^{\eta\gamma}} \lambda_j \circ \pi_Y + \mathfrak{v} ) \circ p_0 \circ \eta) \\
				& \ \ = \ei \del \delb ( M e^{d_j^{\eta \gamma}} \lambda_j \circ p \circ \pi_X  + \mathfrak{v} \circ \eta \circ p_0 ),
			\end{split}
		\end{align*}
		where we used that $\pi_Y \circ \eta = \pi_Y$ and have denoted by $\eta$ both its action on $X_0 \to X$ and on $Y_0 \to Y$. We now finally use \eqref{eq:apflck2} for both terms: $\mathfrak{v}_j^\gamma \circ \eta = c_\eta \mathfrak{v}_j^\gamma$, so $\mathfrak{v} \circ \eta = c_\eta \mathfrak{v}$ and we can continue
		\[
		\eta^* \omega_1 = \ei \del \delb ( M c_\eta e^{d_j^{\gamma}} \lambda_j \circ p \circ \pi_X  + c_\eta \mathfrak{v} \circ p_0 ) = c_\eta \omega_1.
		\]
	\end{proof}

\end{document}